\documentclass{alggeom}

\usepackage{amssymb, amsmath, amsthm}
\usepackage[cmtip,all]{xy}

\newcommand{\Z}{\ensuremath{\mathbb{Z}}}
\newcommand{\Q}{\ensuremath{\mathbb{Q}}}
\newcommand{\C}{\ensuremath{\mathbb{C}}}
\newcommand{\PP}{\ensuremath{\mathbb{P}}}
\newcommand{\M}{\ensuremath{\overline{\mathcal{M}}}}
\newcommand{\OO}{\ensuremath{\mathcal{O}}}
\newcommand{\vir}{\ensuremath{\textrm{vir}}}
\newcommand{\vdim}{\ensuremath{\textrm{vdim}}}
\newcommand{\bt}{\ensuremath{\mathbf{t}}}
\newcommand{\one}{\ensuremath{\mathbf{1}}}
\newcommand{\bbeta}{\ensuremath{\vec{\beta}}}
\newcommand{\tw}{\ensuremath{\textrm{tw}}}
\newcommand{\GIT}{\ensuremath{\mathbin{\mkern-4mu/\mkern-6mu/\mkern-4mu}}}
\newcommand{\HH}{\ensuremath{\mathcal{H}}}
\newcommand{\G}{\ensuremath{\mathcal{G}}}
\newcommand{\te}{\widetilde\epsilon}
\newcommand{\tg}{\widetilde g}
\newcommand{\LL}{\ensuremath{\mathcal{L}}}
\DeclareMathOperator{\val}{val}
\DeclareMathOperator{\ev}{ev}
\DeclareMathOperator{\Hom}{Hom}

\newtheorem{theorem}{Theorem}[section]
\newtheorem{lemma}[theorem]{Lemma}
\newtheorem{conjecture}[theorem]{Conjecture}

\theoremstyle{definition}
\newtheorem{definition}[theorem]{Definition}

\theoremstyle{remark}
\newtheorem{remark}[theorem]{Remark}

\numberwithin{equation}{section}

\begin{document}

\title{Higher-genus quasimap wall-crossing via localization}

\author[E.~Clader]{Emily Clader}
\email{eclader@sfsu.edu}
\address{Department of Mathematics, San Francisco State University, 1600 Holloway Avenue, San Francisco, CA 94132, USA}
\thanks{The first author acknowledges the generous support of Dr.~Max R\"ossler, the Walter Haefner Foundation, and the ETH Foundation; she was also supported by NSF DMS grant 1810968 and NSF CAREER grant 2137060.  The second author was partially supported by NSF grant DMS-2239320, the CNRS and the Swiss National Science Foundation grant SNF 200021\_143274.  The third author was partially supported by NSF grant DMS 1405245 and NSF FRG grant DMS 1159265.}

\author[F.~Janda]{Felix Janda}
\email{fjanda@illinois.edu}
\address{Department of Mathematics, University of Illinois Urbana--Champaign, Urbana, IL 61801, USA}

\author[Y.~Ruan]{Yongbin Ruan}
\email{ruanyb@zju.edu.cn}
\address{Institute for Advanced Study in Mathematics, Zhejiang University, Hangzhou, China}


\begin{abstract}
  We give a new proof of Ciocan-Fontanine and Kim's wall-crossing
  formula relating the virtual classes of the moduli spaces of
  $\epsilon$-stable quasimaps for different $\epsilon$ in any genus,
  whenever the target is a complete intersection in projective space
  and there is at least one marked point.
  Our techniques involve a twisted graph space, which
  can be modified to yield wall-crossing formulas for more general gauged linear
  sigma models.
\end{abstract}

\maketitle

\section{Introduction}

The study of quasimaps was introduced into Gromov--Witten theory several years ago by Ciocan-Fontanine, Kim, and Maulik \cite{CFKM, CFKModuli, MM, Toda}, generalizing the notion of stable maps to a GIT quotient $Z$.  Quasimaps depend on the additional datum of a stability parameter $\epsilon$ varying over positive rational numbers.  When $\epsilon \rightarrow \infty$, they coincide with ordinary stable maps and one recovers the usual Gromov--Witten theory of $Z$, while when $\epsilon \rightarrow 0$, quasimaps are the stable quotients defined by Marian--Oprea--Pandharipande \cite{MOP}.  The latter theory is thought to correspond to the mirror B-model of $Z$ \cite{CZ, CFKBigI, CFKZero}.

When $\epsilon$ varies from $\infty$ to $0$, the theory changes only at certain discrete values, giving a wall-and-chamber structure to the space of stability parameters.  In \cite{CFKZero}, Ciocan-Fontanine and Kim proved a wall-crossing formula relating the genus-zero theories for different values of $\epsilon$.  The higher-genus theory, on the other hand, is well-known to be much more difficult.  Even stating the appropriate generalization of the wall-crossing formula to higher genus is a nontrivial problem.  Ciocan-Fontanine and Kim carried this out in \cite{CFKMirror} and \cite{CFKHigher}, yielding the following remarkable conjecture:

\begin{conjecture}[See \cite{CFKMirror}]
  \label{MainConj}
  Let $Z$ be a complete intersection in projective space, and fix
  $g, n \geq 0$.
  Then
  \begin{align*}
    &\sum_{\beta} q^\beta [\M^{\epsilon}_{g, n}(Z,\beta)]^{\vir} = \\
&\sum_{\beta_0, \beta_1, \ldots, \beta_k} \frac{q^{\beta_0}}{k!} b_{\bbeta*}c_* \left(\prod_{i = 1}^k q^{\beta_i} \ev_{n + i}^*(\mu_{\beta_i}^\epsilon(-\psi_{n + i})) \cap [\M^\infty_{g, n + k}(Z,\beta_0)]^{\vir} \right),
  \end{align*}
  where $\mu^{\epsilon}_{\beta}(z)$ are certain coefficients of the
  $I$-function of $Z$, $b_{\bbeta}$ is a morphism that converts marked
  points to basepoints, and $c$ is the natural contraction morphism
  from $\infty$-stable to $\epsilon$-stable quasimaps.

More generally, let $Z$ be a GIT quotient of the form $W \GIT_{\theta} G$, for $W$ a complex affine variety and $G$ a reductive algebraic group.  Then there is an explicit formula, depending only on coefficients of the $I$-function of $Z$, that relates the virtual fundamental cycles of the moduli spaces of $\epsilon$-stable and $\infty$-stable quasimaps to $Z$.
\end{conjecture}

Ciocan-Fontanine and Kim have proven their conjecture whenever $Z$ is a complete intersection in projective space \cite{CFKMirror}, using virtual push-forward techniques and MacPherson's graph construction.

Quasimap theory is also a special case of the gauged linear sigma model (GLSM), which was recently given a mathematical definition by Fan, Jarvis, and the third author \cite{FJRGLSM}.  More specifically,
for complete intersections $Z$ in projective space, the GLSM has a ``geometric" chamber, which recovers quasimap theory, and a ``Landau--Ginzburg" chamber, which recovers Fan--Jarvis--Ruan--Witten (FJRW) theory when $Z$ is a hypersurface.  Conjecture \ref{MainConj} makes sense in both chambers, so a natural question is whether the same wall-crossing results hold on the Landau--Ginzburg side.

Our primary motivation for studying wall-crossing, in the context both of quasimaps and the Landau--Ginzburg model, is the celebrated Landau--Ginzburg/Calabi--Yau (LG/CY) correspondence.  When $Z$ is a hypersurface, the LG/CY correspondence proposes an explicit connection between Gromov--Witten and FJRW theory, while for more general targets, it can be viewed as a wall-crossing (or ``phase transition") between the different chambers of the GLSM.  In contrast to the wall-crossing expressed by Conjecture \ref{MainConj}, phase transition between chambers involves analytic continuation of generating functions, a much more subtle operation.  Our ultimate goal is to prove the LG/CY correspondence via a series of wall-crossings, first from $\epsilon = \infty$ to $\epsilon = 0$ in the geometric chamber, then across the chamber wall, and finally from $\epsilon = 0$ to $\epsilon = \infty$ in the Landau--Ginzburg chamber.  This strategy has already been carried out in genus zero for hypersurfaces (through a combination of Ciocan-Fontanine--Kim's wall-crossing \cite{CFKZero}, work of Ross and the third author on the Landau--Ginzburg side \cite{RR}, and analytic continuation by Chiodo and the third author \cite{CR}), as well as in genus one for the quintic hypersurface without marked points (by combining the quasimap mirror theorem of Kim--Lho \cite{KimLho} with the work of Guo--Ross \cite{GR2, GR1}).

In the companion paper \cite{CJRLG} to this one, we prove that Conjecture \ref{MainConj} holds in both phases of the GLSM, assuming the existence of at least one marked point.\footnote{Our proof crucially uses that the localization expression for the virtual cycle of the twisted graph space changes in a nontrivial way when the insertions vary, which is why we require the existence of at least one marked point.  See Remark \ref{rem:mp} for further discussion.}  Our techniques are entirely different from Ciocan-Fontanine--Kim's strategy in the geometric phase: we construct a larger moduli space (the ``twisted graph space") with a $\C^*$-action, in which the theories at $\epsilon=\infty$ and at arbitrary $\epsilon$ arise as fixed loci.  This larger moduli space is closely related to the space of mixed-spin $p$-fields considered by Chang--Li--Li--Liu \cite{MSP1, MSP2}; indeed, the two spaces represent different chambers of the same GIT quotient.

The proof of \cite{CJRLG} simplifies considerably in the geometric chamber, due to the absence of certain technical issues regarding the decomposition of cosection-localized virtual cycles.  The result, in this case, is a new proof of Ciocan-Fontanine--Kim's wall-crossing theorem for complete intersections in projective space:

\begin{theorem}
\label{thm:1}
(See Theorem \ref{thm:main})
Conjecture \ref{MainConj} holds whenever $Z$ is a complete intersection in projective space and $n \geq 1$.
\end{theorem}

\noindent The goal of the present paper is to present the proof of Theorem~\ref{thm:1}, in order to illustrate the techniques of \cite{CJRLG} in the simpler and more familiar geometric setting.

\subsection{Related work} Since the completion of this work, Yang Zhou \cite{YangZhou} has given a proof of Conjecture~\ref{MainConj} for any GIT quotient $Z = W \GIT_{\theta} G$ in which $W$ is a complex affine variety with at worst lci singularities, $G$ is a reductive algebraic group acting on $W$, and $\theta$ is a character of $G$ such that $W^{\text{s}}(\theta) = W^{\text{ss}}(\theta)$ is smooth and nonempty; this includes the case of complete intersections in projective space considered here, but notably, it does not require the assumption that $n \geq 1$.  At approximately the same time, Jun Wang also proved the genus-zero wall-crossing formula for hypersurfaces in toric stacks \cite{JunWang} using a twisted graph space inspired by the one used in this paper.

\subsection{Plan of the paper}  In Section \ref{sec:defs}, we review the necessary definitions from quasimap theory and state Theorem~\ref{thm:1} in precise form.  We introduce the twisted graph space in Section \ref{sec:PZ}, and we explicitly compute the contributions to the localization formula from each of its $\C^*$-fixed loci.  Specializing to the case where $Z$ is a point, the twisted graph space is simply $\M_{g,n}(\PP^1,d)$, and in Section \ref{P1}, we study that case in detail.  Finally, in Section \ref{sec:proof}, we turn to the proof of Theorem~\ref{thm:1}.  The idea is to compute, via localization, the pushforward from the twisted graph space to $\M_{g,n}^{\epsilon}(Z,\beta)$ of a certain difference of cohomology classes.  Using the computations of the previous section (which are closely related to the contribution to the localization from degree-zero components), we conclude that this difference changes by an irrational function of the equivariant parameter when an insertion is varied.  This implies that the difference must vanish, and the wall-crossing theorem follows.

\subsection{Acknowledgments}  The notion of a twisted graph space containing $\epsilon$-theory and $\infty$-theory as fixed loci owes an intellectual debt to a number of people.  In particular, the idea that such a space should exist was first mentioned to us by Ionu\c{t} Ciocan-Fontanine, and the space that we define below is related by phase transition to Chang--Li--Li--Liu's moduli space of mixed-spin $p$-fields \cite{MSP1, MSP2}.  We thank all of these authors for their crucial help in inspiring our construction.  We are also grateful to Bumsig Kim and Dustin Ross for many useful conversations and comments.

\section{Definitions and setup}
\label{sec:defs}

Fix a collection of homogeneous polynomials
\begin{equation*}
  W_1, \ldots, W_r \in \C[x_1, \ldots, x_{N+1}]
\end{equation*}
of degrees $d_1, \ldots, d_r$ defining a complete intersection
$Z \subset \PP^N$.
In this section, we review the relevant definitions regarding 
quasimaps to $Z$.
All of these ideas are due to Ciocan-Fontanine, Kim, and Maulik, and we refer the reader to their work for details.

\subsection{Quasimaps and their moduli}

A moduli space of quasimaps depends on the presentation of $Z$ as a
GIT quotient.
In our case, we write
\begin{equation*}
  Z = [AZ \GIT_{\theta} \C^*],
\end{equation*}
where $AZ$ denotes the affine cone over $Z$:
\begin{equation*}
  AZ = \{W_1 = \cdots = W_r = 0\} \subset \C^{N+1}.
\end{equation*}
The $\C^*$-action on $AZ$ is induced by the $\C^*$-action on
$\C^{N + 1}$ with weights $(1, \dotsc, 1)$, and the character
$\theta \in \Hom_{\Z}(\C^*, \C^*) \cong \Z$ is positive.
We use $(x_1, \dotsc, x_{N+1})$ to denote the coordinates on
$\C^{N+1}$.

For each genus $g \geq 0$, nonnegative integer
$n$, and positive rational number $\epsilon$, quasimaps are defined as follows:

\begin{definition}
\label{def:quasimap}
An \emph{$\epsilon$-stable quasimap} to $Z$ consists of an $n$-pointed prestable curve $(C; q_1, \ldots, q_n)$ of genus $g$, a line bundle $L$ on $C$, and a section
\begin{equation*}
\vec{x}  = (x_1, \ldots, x_{N+1}) \in \Gamma(L^{\oplus N+1}),
\end{equation*}
subject to the following conditions:
\begin{itemize}
\item {\it Image}: The sections $\vec{x} = (x_1, \ldots, x_{N+1})$ satisfy the equations
\begin{equation}
  \label{image}
  W_1(\vec{x}) = 0 \in \Gamma(L^{\otimes d_1}),\ \dotsc,\ W_r(\vec{x}) = 0 \in \Gamma(L^{\otimes d_r}).
\end{equation}
\item {\it Nondegeneracy}:  The zero set of $\vec{x}$ is finite and disjoint from the marked points and nodes of $C$, and for each zero $q$ of $\vec{x}$, the order of the zero (that is, the common order of vanishing of $x_1, \ldots, x_{N+1}$) satisfies
\begin{equation}
\label{nondegen}
\text{ord}_q(\vec{x}) \leq \frac{1}{\epsilon}.
\end{equation}
(Zeroes of $\vec{x}$ are referred to as \emph{basepoints} of the
quasimap.)
\item {\it Stability}:  The $\Q$-line bundle
\begin{equation}
\label{stab}
L^{\otimes \epsilon} \otimes \omega_{\log}
\end{equation}
is ample.
\end{itemize}
The {\it degree} of the quasimap is defined as $\beta:=\deg(L)$.
\end{definition}

Note that for $\epsilon > 1$, condition \eqref{nondegen} implies that $x_1, \ldots, x_{N+1}$ have no common zeroes, so $\vec{x}$ defines a map $C \rightarrow \PP^{N}$.
Furthermore, for $\epsilon > 2$, Condition \eqref{stab} says that this map is stable, and condition \eqref{image} says that it lands in the complete intersection $Z$.
Thus, the definition of $\epsilon$-stable quasimap recovers the notion of a stable map to $Z$ in this case.
The degree $\beta$ of the quasimap corresponds to the degree of the composition $C \to Z \to \PP^N$.
To emphasize their role in the more general theory, we refer to stable maps as $\infty$-stable quasimaps in what follows.

For $\epsilon \leq \frac{1}{\beta}$, on the other hand, condition \eqref{nondegen} puts no restriction on the orders of the basepoints, and \eqref{stab} is equivalent to imposing the analogous requirement for all $\epsilon > 0$.  The resulting objects are stable quotients \cite{MOP}, which are sometimes referred to as $(0+)$-stable quasimaps.

\begin{remark}
  \label{rem:rational}
  An alternative way to view the choice of stability parameter
  $\epsilon$ is to replace the choice of character
  $\theta \in \Hom_{Z}(\C^*, \C^*)$ in the GIT quotient by a choice of
  rational character---that is,
  $\theta \in \Hom_{\Z}(\C^*, \C^*) \otimes_{\Z} \Q \cong \Q$.
  This perspective is useful in what follows.
\end{remark}

The first key foundational result about quasimaps is the following:

\begin{theorem}[Ciocan-Fontanine--Kim--Maulik \cite{CFKM}]
  There is a proper, separated Deligne--Mumford moduli stack
  $\M^{\epsilon}_{g,n}(Z,\beta)$ parameterizing genus-$g$,
  $n$-pointed, $\epsilon$-stable quasimaps of degree $\beta$ to $Z$ up
  to isomorphism.
\end{theorem}

Moreover, in \cite[Section 4.5]{CFKM}, Ciocan-Fontanine, Kim, and Maulik exhibit
a perfect obstruction theory on $\M^{\epsilon}_{g,n}(Z,\beta)$
relative to the stack $\mathcal{D}_{g,n,\beta}$ of curves
equipped with a line bundle of degree $\beta$.
This can be used to define a virtual cycle
\begin{equation*}
[\M^{\epsilon}_{g,n}(Z,\beta)]^{\vir} \in A_{\vdim}(\M^{\epsilon}_{g,n}(Z,\beta)),
\end{equation*}
where
\begin{equation*}
\vdim(\M^{\epsilon}_{g,n}(Z,\beta)) = (N-r-3)(1-g) + \beta (N + 1 - \sum_{i = 1}^r d_i) + n.
\end{equation*}
Here, we recall that $Z$ is a complete intersection of degrees $d_1, \dotsc, d_r$ in $\PP^N$.

\begin{remark}
  \label{rem:genuszero}
  In genus zero, the definition of the virtual cycle simplifies
  substantially.
  Indeed, the genus-zero virtual cycle is simply the Euler class of a
  vector bundle; this is exactly analogous to the situation for stable
  maps.
\end{remark}

There are evaluation maps
\begin{equation*}
\ev_i\colon \M^{\epsilon}_{g,n}(Z,\beta) \rightarrow Z
\end{equation*}
for each $i=1, \ldots, n$, since the basepoints of $\vec{x}$ do not occur at marked points.  Furthermore, psi classes $\psi_1, \ldots, \psi_n$ can be defined just as in the usual theory of stable maps.  Using these, one can define quasimap correlators by analogy to the usual Gromov--Witten invariants.  We denote
\begin{equation*}
\mathcal{H} := H^*(Z;\C),
\end{equation*}
which is the state space of the theory (that is, the vector space from which insertions to the correlators are drawn), and
\begin{equation*}
H:= c_1(\OO_Z(1)),
\end{equation*}
the hyperplane class.

\subsection{The $J$-function}
\label{J}

The small $J$-function for $\epsilon$-stable quasimap theory was defined by Ciocan-Fontanine and Kim in \cite{CFKZero}, generalizing its original definition by Givental \cite{GiventalElliptic, GiventalMirror, GiventalEquivariant} in Gromov--Witten theory and building on the interpretation in terms of contraction maps due to Bertram \cite{Bertram}.  We recall their definition below.

Let $\G \M^{\epsilon}_{0,k}(Z,\beta)$ denote a ``graph space'' version
of $\M^{\epsilon}_{0,k}(Z,\beta)$.
That is, $\G\M^{\epsilon}_{0,k}(Z,\beta)$ parameterizes the same data
as the original moduli space, together with the additional datum of a
degree-$1$ map $C \rightarrow \PP^1$ (or equivalently, a
parametrization of one component of $C$), and the ampleness condition
\eqref{stab} is not required on the parameterized component.
Denote the parameterized component by $C_0 \cong \PP^1$, with
coordinates $[z_1:z_2]$.
Let $\C^*$ act on $\G \M^{\epsilon}_{0,k}(Z,\beta)$ by multiplication
on $z_2$, and let $z$ denote the weight at the tangent space of
$0 = [0 : 1] \in \PP^1$.

The fixed loci of this action consist of quasimaps for which all of marked points and basepoints and the entire degree $\beta$ is concentrated over $0$ and $\infty$ in $C_0$.  We denote by $F^{\epsilon}_{\beta} \subset \G \M^{\epsilon}_{0,k}(Z,\beta)$ the fixed locus on which everything is concentrated over $0$.  More precisely, when $k\geq 1$ or $\beta > 1/\epsilon$, an element of $F^{\epsilon}_{\beta}$ consists of an $\epsilon$-stable quasimap to $Z$ attached at a single marked point to $C_0$, so
\begin{equation*}
F^{\epsilon}_{\beta} \cong \M^{\epsilon}_{0,k+1}(Z,\beta).
\end{equation*}
When $k=0$ and $\beta \leq 1/\epsilon$, on the other hand, such a quasimap would not be stable; instead, $C_0$ is the entire source curve, and the quasimap has a single basepoint of order $\beta$ at $0$.  In either case, there is an evaluation map
\begin{equation*}
\ev_{\bullet}\colon F^{\epsilon}_{\beta} \rightarrow Z,
\end{equation*}
defined by evaluation at $\infty \in C_0$.

Following \cite{CFKZero}, we define:

\begin{definition}
  Let $q$ be a formal Novikov variable. 
  The {\it $\epsilon$-stable $J$-function} is defined by
  \begin{equation*}
    J^{\epsilon}(q,\bt,z):= z\sum_{k \geq 0, \beta \geq 0} q^{\beta} (\ev_{\bullet})_* \left( \frac{\prod_{i=1}^k \ev_i^*(\bt(\psi_i))}{k!}  \cap \frac{[F_{\beta}^{\epsilon}]^{\vir}}{e_{\C^*}\left(N^{\vir}_{F_{\beta}^{\epsilon}/\G \M^{\epsilon}_{0,k}(Z,\beta)}\right)} \right),
  \end{equation*}
  where $\bt \in \HH[\![z]\!]$.
  (Note that our convention differs by an overall factor of $z$ from
  \cite{CFKZero}.)
\end{definition}

The {\it small $\epsilon$-stable $J$-function} is
\begin{equation*}
J^{\epsilon}(q,\bt = 0, z) =: J^{\epsilon}(q,z).
\end{equation*}
More explicitly, we have
\begin{equation*}
  J^{\epsilon}(q,z) = z + \sum_{\beta > 1/\epsilon} q^{\beta}  (\ev_{\bullet})_*\left( \frac{[\M^{\epsilon}_{0,\bullet}(Z,\beta)]^{\vir}}{z-\psi_{\bullet}}\right)+\text{unstable terms}.
\end{equation*}
The ``unstable terms" are the terms with $\beta \leq 1/\epsilon$, and can be computed explicitly, as explained in \cite{Bertram}.  In particular, taking $\epsilon \rightarrow 0+$ (that is, requiring the stability condition \eqref{stab} for all $\epsilon >0$), every term of the $J$-function becomes unstable, and one obtains a function $I(q,z) = J^{0+}(q,z)$ that can be calculated exactly.  This is the $I$-function of $Z$, as studied by Givental \cite{GiventalMirror} and many others:
\begin{equation}
\label{I}
I(q,z) = z\sum_{\beta \geq 0} q^{\beta} \frac{\prod_{i=1}^r \prod_{b=1}^{d_i\beta} (d_iH + bz)}{ \prod_{b=1}^{\beta} (H + bz)^{N+1}}.
\end{equation}
More generally, truncating \eqref{I} to powers of $q$ less than or equal to $1/\beta$ yields the unstable part of $J^{\epsilon}$ for any $\epsilon$.

We denote by
\begin{equation*}
  [J^\epsilon]_+(q,z) \in \mathcal{H}[\![q, z]\!]
\end{equation*}
the part of the $J$-function with non-negative powers of $z$, and we
let $\mu^\epsilon_\beta(z)$ denote the $q^\beta$-coefficient in
$-z\one + [J]^\epsilon_+(q,z)$:
\begin{equation*}
  \sum_\beta q^\beta \mu_\beta^\epsilon(z) = -z\one + [J^\epsilon]_+(q,z) = O(q).
\end{equation*}
This is sometimes referred to as the ``mirror transformation."

\subsection{The wall-crossing conjecture}

The genus-zero wall-crossing conjecture states that the function
$J^\epsilon (q, z)$ lies on the Lagrangian cone of the Gromov--Witten
theory of $Z$, or in other words, that there exists
$\bt \in \HH[\![z]\!]$ such that $J^\epsilon (q, z) = J^\infty (q, \bt, z)$.  More explicitly, the requisite $\bt$ is determined by the fact that $J^{\infty}(q,\bt,z) = z\one + \bt(-z) + O(z^{-1})$, so the fact that $J^\epsilon (q, z) = J^\infty (q, \bt, z)$ says that
\begin{equation*}
  J^\epsilon (q, z)
  = J^\infty \left(q, z\one +[J^\epsilon]_+ (q, -z), z\right)
  = J^\infty \left(q, \sum_\beta q^\beta \mu^\epsilon_\beta(-z), z\right).
\end{equation*}
Thus, the genus-zero wall-crossing conjecture can be rephrased as the agreement of the functions $J^\epsilon$ and $J^\infty$ up to a shift by the mirror transformation.

More generally, the wall-crossing conjecture in any genus can be stated on the level of virtual cycles.  To do so, we require a bit more notation.  For a tuple of nonnegative integers
$\bbeta = (\beta_1, \ldots, \beta_k)$, there is a morphism
\begin{equation*}
  b_{\bbeta}\colon \M^{\epsilon}_{g,n+k}\left(Z, \beta- \sum_{i=1}^k \beta_i\right) \rightarrow \M^{\epsilon}_{g,n}(Z, \beta)
\end{equation*}
that ``trades" the last $k$ marked points for basepoints of orders
$\beta_1, \ldots, \beta_k$.
Furthermore, there is a morphism
\begin{equation*}
  c\colon \M^{\infty}_{g,n}(Z, \beta) \rightarrow \M^{\epsilon}_{g,n}(Z,\beta)
\end{equation*}
that contracts rational tails and replaces them with basepoints.
See \cite[Sections 3.1 and 3.2]{CFKZero} and \cite[Section
1.4]{CFKMirror} for more careful definitions.

We are now ready to state our main theorem, a verification of the all-genus wall-crossing conjecture, in precise form:
\begin{theorem}
  \label{thm:main}
  Let $Z$ be a complete intersection in projective space, and fix $g\geq0$ and $n \geq 1$.  Then
  \begin{multline*}
    \sum_{\beta} q^\beta [\M^{\epsilon}_{g, n}(Z,\beta)]^{\vir} = \\
    \sum_{\beta_0, \beta_1, \ldots, \beta_k} \frac{q^{\beta_0}}{k!} b_{\bbeta*} c_* \left(\prod_{i = 1}^k q^{\beta_i} \ev_{n + i}^*(\mu_{\beta_i}^\epsilon(-\psi_{n + i})) \cap [\M^\infty_{g, n + k}(Z,\beta_0)]^{\vir} \right).
  \end{multline*}
\end{theorem}
\begin{remark}
  Theorem~\ref{thm:main} proves wall-crossing from $\infty$-stability
  to $\epsilon$-stability.
  Also, if we are given two stability parameters
  $\epsilon_+ > \epsilon_-$, Theorem~\ref{thm:main} implies
  wall-crossing from $\epsilon^+$-stability to $\epsilon^-$-stability
  as in \cite[Conjecture~1.1]{CFKMirror}.
  This can be seen by applying Theorem~\ref{thm:main} to both
  stability conditions, and then taking the difference.
\end{remark}

\begin{remark}
  As explained in \cite[Corollary 1.5]{CFKMirror}, Theorem
  \ref{thm:main} implies a comparison of the genus-$g$ generating
  functions of $\epsilon$-theory and $\infty$-theory with $n \geq 1$.
\end{remark}

\subsection{Twisted theory}
\label{subsec:twisted}

Our proof of Theorem \ref{thm:main} requires a generalization of the theory of $\epsilon$-stable quasimaps, defined via a twist by an equivariant Euler class.

Let
\begin{equation*}
  \pi\colon \mathcal{C} \rightarrow \M^{\epsilon}_{g,n}(Z,\beta)
\end{equation*}
denote the universal curve, and let $\mathcal{L}$ denote the universal
line bundle on $\mathcal{C}$.
Consider a trivial $\C^*$-action on $\M^{\epsilon}_{g,n}(Z,\beta)$,
lifted to an action scaling the fibers of $R\pi_*(\mathcal{L}^{\vee})$
with weight $-1$.
We write $\C_{(\lambda)}$ for a nonequivariantly trivial line bundle
with a $\C^*$-action of weight $1$, where $\lambda$ denotes the
equivariant parameter of the action.
To distinguish this action from the one on the graph space we write
$\C^*$ as $\C^*_{\lambda}$.

Define
\begin{equation*}
  [\M^{\epsilon}_{g,n}(Z,\beta)]^{\vir}_\tw := \frac{[\M^{\epsilon}_{g,n}(Z,\beta)]^{\vir}}{e_{\C^*_{\lambda}}\left(R\pi_*\left(\mathcal{L}^{\vee} \otimes \C_{(\lambda)}\right)\right)}.
\end{equation*}
Using this, we define a twisted $J$-function by
\begin{align*}
  J^\epsilon_\tw(q,z) &:= -z^2 e_{\C^*_{\lambda}}(\OO_Z(-1) \otimes \C_{(\lambda)})\times\\
  &\sum_{\beta} q^{\beta} (\ev_{\bullet})_* \left(\frac{[\M^{\epsilon}_{0,\bullet}(Z,\beta)]^{\vir} \cap e_{\C^*_{\lambda} \times \C^*_z}\left(-R\pi_*\left(\mathcal L^{\vee}\otimes \C_{(\lambda)}\right)\right)}{e_{\C^*_z}(N_{F^{\epsilon}_{\beta}})} \right).
\end{align*}
Here, $\C^*_z$ denotes the $\C^*$-action on the $z_2$-coordinate of
the graph space.  From the twisted $J$-function we define the twisted mirror
transformation:
\begin{equation}
\label{mutw}
  \sum_\beta q^\beta \mu_\beta^{\epsilon, \tw}(z) = -\one z + [J^\epsilon_\tw(z)]_+.
\end{equation}
In genus zero, there is again a wall-crossing conjecture of the form
\begin{equation}
  \label{JWC}
  J^\epsilon_\tw (q, z)
  = J^\infty_\tw \left(q, \sum_\beta q^\beta \mu_\beta^{\epsilon, \tw}(-z), z\right).
\end{equation}
In our situation, \eqref{JWC} has been proven in \cite[Corollary~7.3.2]{CFKZero}.

Equipped with these definitions, we can state a twisted
wall-crossing theorem, in direct analogy to Theorem~\ref{thm:main}:
\begin{theorem}
  \label{thm:twisted}
  Let $Z$ be a complete intersection in projective space, and fix
  $g\geq0$ and $n \geq 1$.
  Then
  \begin{multline*}
    \sum_{\beta} q^\beta [\M^{\epsilon}_{g, n}(Z,\beta)]^{\vir}_{\tw} = \\
    \sum_{\beta_0, \beta_1, \ldots, \beta_k} \frac{q^{\beta_0}}{k!} b_{\bbeta*} c_* \left(\prod_{i = 1}^k q^{\beta_i} \ev_{n + i}^*(\mu_{\beta_i}^{\epsilon, \tw}(-\psi_{n + i})) \cap [\M^\infty_{g, n + k}(Z,\beta_0)]^{\vir}_{\tw} \right).
  \end{multline*}
\end{theorem}

The key point, now, is the following:
\begin{lemma}
\label{twuntw}
The twisted wall-crossing theorem (Theorem \ref{thm:twisted}) implies
the untwisted wall-crossing theorem (Theorem \ref{thm:main}).
\begin{proof}
  For each $\beta$, the coefficients of $q^{\beta}$ on the two sides
  of Theorem~\ref{thm:twisted} are Laurent polynomials in the
  equivariant parameter $\lambda$.
  Then
  \begin{equation*}
    [\M^{\epsilon}_{g, n}(Z,\beta)]^{\vir}_{tw} = \sum_{i \leq g - 1 - \beta} \lambda^i C_i \cap [\M^{\epsilon}_{g, n}(Z,\beta)]^{\vir}
  \end{equation*}
  for nonequivariant classes
  $C_i \in H^{g - 1 - \beta - i}(\M^{\epsilon}_{g, n}(Z,\beta))$ for which the
  top coefficient is $C_{g - 1 - \beta} = \one$.

  The untwisted wall-crossing follows from Theorem~\ref{thm:twisted}
  by taking the coefficient of $\lambda^{g - 1 - \beta}$ on both
  sides: the left-hand side directly yields $[\M^{\epsilon}_{g, n}(Z,\beta)]^{\vir}$,
  while the only contribution on the right-hand side comes from taking
  the top power of $\lambda$ in every factor of $\mu^{\epsilon, \tw}$ and in
  $[\M^{\infty}_{g, n}(Z,\beta)]^{\vir}_{\tw}$, thus producing the right-hand
  side of Theorem~\ref{thm:main}.
\end{proof}
\end{lemma}

\section{Twisted graph space}
\label{sec:PZ}

The proof of Theorem~\ref{thm:main} proceeds by $\C^*$-localization on an enhanced version $\PP Z^{\epsilon}_{g,n,\beta,d}$ of the moduli space of $\epsilon$-stable quasimaps.  We refer to it as the ``twisted graph space", due to its resemblance to the graph space described in Section \ref{J}.  Its key feature is that quasimaps landing at one fixed point are $\epsilon$-stable and quasimaps landing at the other fixed point are $\infty$-stable.

\subsection{Definition of the twisted graph space}
\label{sec:PZ:def}

The twisted graph space consists of quasimaps to the projectivization of the vector bundle $\OO_Z(-1) \oplus \OO_Z$ on $Z$, which is the $\PP^1$-bundle
\begin{equation*}
  \PP Z := (AZ \times \C^2)\GIT_\theta (\C^*)^2
\end{equation*}
over $Z$, in which $AZ$ again denotes the affine cone over $Z$, the action of $(\C^*)^2$ is
\begin{equation*}
  (g,t) \cdot (\vec{x}, z_1, z_2) = (g \vec{x}, g^{-1}tz_1, tz_2),
\end{equation*}
and the (rational) character
$\theta\colon (\C^*)^2 \rightarrow \C^*$ is
\begin{equation*}
  \theta(g,t) = g^{\epsilon} t^3.
\end{equation*}
(See Remark \ref{rem:rational} above for the meaning of a GIT quotient
with rational character.) 

In more concrete terms, we have
\begin{equation*}
  \PP Z^{\epsilon}_{g,n,\beta,d} = \{(C; q_1, \ldots, q_n; L_1, L_2, \sigma)| \dots\}/\sim ,
\end{equation*}
where $(C;q_1, \ldots, q_n)$ is an $n$-pointed prestable curve of
genus $g$, $L_1$ and $L_2$ are line bundles on $C$ with
\begin{equation*}
  \beta = \deg(L_1), \; \; \; d = \deg(L_2),
\end{equation*}
and
\begin{equation*}
  \sigma \in \Gamma(L_1^{\oplus N + 1}  \oplus (L_1^\vee \otimes L_2) \oplus L_2).
\end{equation*}
We denote the components of $\sigma$ by $(\vec x, z_1, z_2)$.
The following conditions are required:
\begin{itemize}
\item {\it Image}: The sections $\vec{x} = (x_1, \ldots, x_{N+1})$ satisfy the equations
  \begin{equation*}
    W_1(\vec{x}) = 0 \in \Gamma(L^{\otimes d_1}), \dotsc, W_r(\vec{x}) = 0 \in \Gamma(L^{\otimes d_r}).
  \end{equation*}
\item {\it Nondegeneracy}\footnote{The nondegeneracy condition here, though cumbersome to write down, is simply the explicit statement of the length condition appearing in \cite[Section 2.1]{CFKBigI}.}: The zero set of $\vec{x}$ is finite and disjoint from the marked points and nodes of $C$.
  The sections $z_1$ and $z_2$ never simultaneously vanish.  Furthermore, for each point $q$ of $C$ at which $z_2(q)\neq 0$, we have
  \begin{equation*}
    \operatorname{ord}_q(\vec x) \le \frac{1}{\epsilon},
  \end{equation*}
  and for each point $q$ of $C$ at which $z_2(q) = 0$ (and hence $z_1(q) \neq 0$), we have
    \begin{equation*}
    \operatorname{ord}_q(\vec x) =0.
      \end{equation*}
\item {\it Stability}: The $\Q$-line bundle
  \begin{equation*}
    L_1^\epsilon \otimes L_2^{\otimes 3} \otimes \omega_{\log}
  \end{equation*}
  is ample.
\end{itemize}

There are natural evaluation maps
\begin{equation*}
  \ev_i\colon \PP Z^{\epsilon}_{g, n, \beta, d} \to \PP Z,
\end{equation*}
given, as before, by evaluating the sections $\sigma$ at $q_i$.
Moreover, since $\PP Z^{\epsilon}_{g,n,\beta,d}$ is a moduli space of
stable quasimaps to an lci GIT quotient, the results of \cite{CFKM}
imply that it is a proper Deligne--Mumford stack equipped with a
natural perfect obstruction theory relative to the stack
$\mathcal{D}_{g,n,\beta,d}$ of curves equipped with a pair of line
bundles.
This obstruction theory is of the form
\begin{equation}
\label{pot}
R\pi_*(u^*\mathbb{RT}_{\rho}).
\end{equation}
Here, we denote the universal family over $\PP Z^\epsilon_{g,n,\beta,d}$ by
\begin{equation*}
\xymatrix{
\mathcal{V}\ar_{\rho}[r] & \mathcal{C} \ar^{\pi}[d]\ar_{u}@/_{1pc}/[l]\\
& \PP Z^{\epsilon}_{g,n,\beta,d},
}
\end{equation*}
where $\LL_1$ and $\LL_2$ are the universal line bundles, and where
\begin{equation*}
  \mathcal{V} \subset \LL_1^{\oplus N+1} \oplus (\LL_1^{\vee} \otimes \LL_2) \oplus \LL_2
\end{equation*}
is the subcone defined by the vanishing of
$W_1(\vec x), \dotsc, W_r(\vec x)$.
Furthermore, we use $\mathbb{RT}_{\rho}$ to denote the relative
tangent complex of the complete intersection morphism $\rho$.
Somewhat more explicitly, \eqref{pot} equals
\begin{equation}
\label{pot2}
\mathbb{E} \oplus R\pi_*((\LL_1^{\vee} \otimes \LL_2) \oplus \LL_2),
\end{equation}
in which the summand $\mathbb{E}$ comes from the deformations and
obstructions of the sections $\vec{x}$.  Because the sections $z_1$ and $z_2$ never simultaneously vanish, they define a section $f\colon C \to \PP(L_1^\vee \oplus \OO)$.
Using this observation, we can rewrite the obstruction theory to be
more similar to the standard obstruction theory for stable maps.
More precisely, the complex
\begin{equation}
  \label{pot3}
  \mathbb{E} \oplus R\pi_* \left(f^* T_{\PP(\LL_1^\vee \oplus \OO)/\mathcal{C}}\right),
\end{equation}
is a perfect obstruction theory on $\PP Z^{\epsilon}_{g,n,\beta,d}$
relative to the stack $\mathcal{D}_{g,n,\beta}$ of curves equipped
with only one line bundle.

The virtual dimension of $\PP Z^{\epsilon}_{g,n,\beta,d}$ is equal to
\begin{equation*}
  \text{vdim}(\PP Z^{\epsilon}_{g, n, \beta, d}) = \text{vdim}(Z^{\epsilon}_{g, n, \beta}) + 2d - \beta + 1 - g.
\end{equation*}

\subsection{$\C^*$-action and fixed loci}
\label{fixedloci}

Let $\C^*$ act on $\PP Z^{\epsilon}_{g,n,\beta,d}$ with weight $-1$ on
the $z_1$-coordinate of $\sigma$.
Equivalently and most conveniently for the further computations, we
can introduce the $\C^*$-action by viewing $z_1$ now as a section of
$L_1^\vee \otimes L_2 \otimes \C_{(\lambda)}$, where $\C_{(\lambda)}$
denotes the non-equivariantly trivial line bundle of $\C^*$-weight
$1$.
We set $\lambda := c_1(\C_{(\lambda)})$; the localization
  computations that follow show that this use of $\lambda$ is
  consistent with Section \ref{subsec:twisted}.
From now on, we regard $f$ as a map to
$\PP(\mathcal L_1^\vee \otimes \C_{(\lambda)} \oplus \OO)$.
Under the replacement
$\LL_1 \leadsto \widetilde\LL_1 := \LL_1 \otimes \C_{(-\lambda)}$, the
discussion from Section~\ref{sec:PZ:def} evidently yields a
$\C^*$-equivariant perfect obstruction theory.

Analogously to the graph space $\G \M^{\epsilon}_{0, n}(Z,\beta)$, as well as to
the well-known situation in Gromov--Witten theory, the fixed loci of
the $\C^*$-action on $\PP Z^{\epsilon}_{g, n, \beta, d}$ are indexed by certain
decorated graphs.
A graph $\Gamma$ consists of vertices, edges, and $n$ legs, with the
following decorations:
\begin{itemize}
\item Each vertex $v$ is decorated by an index
  $j(v) \in \{0, \infty\}$, a genus $g(v)$, and a degree
  $\beta(v) \in \mathbb{N}$.
\item Each edge $e$ is decorated by a degree $d(e) \in \mathbb{N}$.
\item Each leg is decorated by an element of $\{1, 2, \ldots, n\}$.
\end{itemize}
The valence $\val(v)$ of a vertex $v$ denotes the total number
of incident legs and half-edges.

The fixed locus in $\PP Z^{\epsilon}_{g, n, \beta, d}$ indexed by the decorated
graph $\Gamma$ parameterizes quasimaps of the following type:
\begin{itemize}
\item Each edge $e$ corresponds to a genus-zero component $C_e$ on which 
  \begin{equation*}
    \deg(L_2|_{C_e}) = d(e),
  \end{equation*}
  and which has two distinguished ``ramification points" $q_1$, $q_2$.
  The section $z_1$ vanishes nowhere except possibly at $q_1$, while
  $z_2$ vanishes only at $q_2$, and all of the marked points, all of
  the nodes, and all of the degree of $L_1|_{C_e}$ is concentrated at
  $q_1$ and $q_2$.
  That is,
  \begin{equation*}
    \deg(L_1|_{C_e}) = \text{ord}_{q_1}(\vec{x}) + \text{ord}_{q_2}(\vec{x}),
  \end{equation*}
  so if both ramification points are special points, it follows that the degree of $L_1|_{C_e}$ is zero.
\item Each vertex $v$ for which $j(v) = 0$ (with certain unstable
  exceptional cases noted below) corresponds to a maximal sub-curve
  $C_v$ of $C$ over which $z_1 \equiv 0$, and each vertex $v$ for
  which $j(v) = \infty$ (again with some exceptions) corresponds to a
  maximal sub-curve over which $z_2 \equiv 0$.
  The label $g(v)$ denotes the genus of $C_v$, the label $\beta(v)$
  denotes the degree of $L_1|_{C_v}$, and the legs incident to $v$
  indicate the marked points on $C_v$.
\item A vertex $v$ is unstable if stable quasimaps of the type described
  above do not exist.  In this case, $v$ corresponds to a single point of the component
  $C_e$ for each adjacent edge $e$, which may be a node at which $C_e$
  meets $C_{e'}$, a marked point of $C_e$, or a basepoint on $C_e$ of
  order $\beta(v)$.
  (It is possible to have $\beta(v) = 0$, in which case $v$ corresponds to a nonspecial point that is not a basepoint.)
\end{itemize}

Observe that for a stable vertex $v$ such that $j(v) = 0$, we have
$z_1|_{C_v} \equiv 0$, so the nondegeneracy condition implies that
$\operatorname{ord}_q(\vec x) < 1/\epsilon$ for each $q \in C_v$. 
That is, the restriction of $(C; q_1, \ldots, q_n; L_1; \vec x)$ to
$C_v$ defines an element of
$\M^{\epsilon}_{g(v), \val(v)}(Z,\beta(v))$. 
On the other hand, for a stable vertex $v$ such that $j(v) = \infty$,
we $z_2|_{C_v} \equiv 0$, so the non-degeneracy condition implies that
$\vec x$ is nowhere-vanishing on $C_v$. 
Thus, the restriction of $(C; q_1, \ldots, q_n; L_1; \vec x)$ to $C_v$
defines an element of $\M^{\infty}_{g(v), \val(v)}(Z,\beta(v))$.
Finally, for an edge $e$, the restriction of $\vec{x}$ to $C_e$
defines a constant map to $Z$ (possibly with an additional basepoint
at the ramification point where $z_1 = 0$) because otherwise
$(C_e; n_1, n_2; L_1|_{C_e}; \vec x|_{C_e})$ (where $n_1$ and $n_2$
are the nodes on $e$) has no automorphisms so that the section
$C_e \to \PP(L_1^\vee \oplus \OO)$ cannot be $\C^*$-fixed.

\begin{remark}
\label{rem:dbeta}
We observe that if $C_e$ is an edge component containing a basepoint of order $\beta(e)$, then one must have $d(e) \ge \beta(e)$, since otherwise $z_1 \equiv 0$ and, given that $z_2$ must vanish somewhere, this is impossible without violating the nondegeneracy condition in the definition of the twisted graph space.
\end{remark}

Denote by $F_{\Gamma}$ the moduli space
\begin{equation*}
  \prod_{\substack{v \text{ stable}\\ j(v) = 0}} \M^{\epsilon}_{g(v), \val(v)}(Z,\beta(v)) \times_Z \prod_{\text{edges } e} Z^{\frac 1{d(e)}} \times_Z \prod_{\substack{v \text{ stable}\\ j(v) = \infty}} \M^{\infty}_{g(v), \val(v)}(Z,\beta(v)),
\end{equation*}
in which the fiber product over $Z$ imposes that the evaluation maps
at the two branches of each node agree.
Here, $Z^{\frac 1d}$ denotes the $d$th root stack over $Z$
corresponding to the line bundle $\OO(-1) \otimes \C_{(\lambda)}$; the appearance of the $d$th root stack is explained in Section~\ref{sec:loc:edge}.
The discussion of the previous paragraph implies that there is a
canonical family of $\C^*$-fixed elements of
$\PP Z^{\epsilon}_{g, n, \beta, d}$ over $F_{\Gamma}$, yielding a
morphism
\begin{equation*}
  \iota_{\Gamma}\colon F_{\Gamma} \rightarrow \PP Z^{\epsilon}_{g, n, \beta, d}.
\end{equation*}
This is not exactly the inclusion of the associated fixed locus,
because elements of $\PP Z^{\epsilon}_{g, n, \beta, d}$ have additional
automorphisms from permuting the components $C_v$ via an automorphism
of $\Gamma$.
Nevertheless, there is a finite map from $F_{\Gamma}$ to the fixed
locus whose degree can be explicitly calculated, and this is
sufficient for our purposes.

\subsection{Localization contributions}

The virtual localization formula of Graber--Pandharipande \cite{GP}
and improved by Chang--Kiem--Li \cite{CKL} expresses
$[\PP Z^{\epsilon}_{g, n, \beta, d}]^{\vir}$ in terms of contributions
from each fixed-locus graph $\Gamma$:
\begin{equation}
  \label{localization}
  [\PP Z^{\epsilon}_{g, n, \beta, d}]^{\vir} = \sum_{\Gamma} \frac{1}{|\text{Aut}(\Gamma)|}\iota_{\Gamma*} \left(\frac{[F_{\Gamma}]^{\vir}}{e(N_{\Gamma}^{\vir})}\right).
\end{equation}
Here, $[F_{\Gamma}]^{\vir}$ is computed via the $\C^*$-fixed part of
the restriction to the fixed locus of the obstruction theory on
$\PP Z^{\epsilon}_{g,n,\beta,d}$, and $e(N_{\Gamma}^{\vir})$ as the
equivariant Euler class of the $\C^*$-moving part of this restriction.
We note that when pushing forward this expression to
$Z^\epsilon_{g, n, \beta}$, the integration over the factors
$Z^{\frac 1{d(e)}}$ yields an additional global factor of
$\prod_e (d(e))^{-1}$.

The goal of this subsection is to compute the contributions of each
graph $\Gamma$ explicitly.
In order to compute the contribution of a graph $\Gamma$ to
\eqref{localization}, one must first apply the normalization exact
sequence to the relative obstruction theory \eqref{pot2}, thus
breaking the contribution of $\Gamma$ to \eqref{localization} into
vertex, edge, and node factors.
This accounts for all but the automorphisms and deformations within
$\mathcal{D}_{g,n,\beta,d}$.
The latter come from deformations of the vertex components and their
line bundles, deformations of the edge components and their line
bundles, and deformations smoothing the nodes; these are included in
the vertex, edge, and node contributions, respectively, in what
follows.  We include the factors from automorphisms of the source curve also in
the edge contributions.

\subsubsection{Vertex contributions}

Let $v$ be a stable vertex corresponding to a sub-curve $C_v$ via $f$
contracted to the $j(v)$-section of
$\PP(\widetilde\LL_1^\vee|_{C_v} \oplus \OO)$.
The sections $\vec{x}$ are $\C^*$-fixed, and the deformations of these
sections, together with the deformations of $C_v$ and the line bundle
$\LL_1|_{C_v}$, contribute the virtual cycle
$[\M^{\epsilon}_{g(v),\val(v)}(Z,\beta(v))]^{\vir}$ or
$[\M^{\infty}_{g(v),\val(v)}(Z,\beta(v))]^{\vir}$ on the vertex
moduli, when $j(v) = 0$ or $j(v) = \infty$, respectively.

On the other hand, the line bundle
$f^* T_{\PP(\widetilde\LL_1^\vee|_{C_v} \oplus \OO)/\mathcal C}$ is
$\C^*$-moving.
The restriction of
$T_{\PP(\widetilde\LL_1^\vee|_{C_v} \oplus \OO)/\mathcal C}$ to the zero
section is $\widetilde\LL_1^\vee|_{C_v}$, while its restriction to the
infinity section is $\widetilde\LL_1|_{C_v}$.  Thus, in total a stable vertex $v$ with $j(v) = 0$ contributes
\begin{equation*}
  \frac{[\M^{\epsilon}_{g(v),\val(v)}(Z,\beta(v))]^{\vir}}{e^{\C^*}\left(R\pi_*\left(\LL^{\vee} \otimes \C_{(\lambda)}\right)\right)},
\end{equation*}
while a stable vertex $v$ with $j(v) = \infty$ contributes
\begin{equation*}
  \frac{[\M^{\infty}_{g(v),\val(v)}(Z, \beta(v))]^{\vir}}{e^{\C^*}\left(R\pi_*\left(\LL \otimes \C_{(-\lambda)}\right)\right)}
\end{equation*}
to $[F_{\Gamma}]^{\vir}/e(N_{\Gamma}^{\vir})$, where $\LL$ is the
universal line bundle.

\subsubsection{Edge contributions}
\label{sec:loc:edge}

Consider an edge $C_e$ that has a special point at each of its
ramification points.
Then the edge moduli in $F_{\Gamma}$ is $Z^{\frac 1{d(e)}}$,
parameterized by the image of the constant morphism induced by
$\vec{x}$.
Let us abbreviate $\widetilde Z := Z^{\frac 1{d(e)}}$ in this section.
By its universal property, the root stack $\widetilde Z$ possesses a $d(e)$th
root $\mathcal R := \OO_{\widetilde Z}(-1/d(e)) \otimes \C_{(\lambda/d(e))}$ of
$\OO_{\widetilde Z}(-1) \otimes \C_{(\lambda)}$.
The universal family over the edge moduli space is
\begin{equation}
  \label{universalfam}
  \xymatrix{\mathcal{C}_e:= \PP\left(\mathcal R \oplus \OO_{\widetilde Z}\right) \ar[r]^-{f}\ar[d]_{\pi} & \PP(\widetilde\LL_1^\vee|_{C_e} \oplus \OO_{\widetilde Z}) \\
    \M_e := \widetilde Z, &}
\end{equation}
where $f$ is given in coordinates on the $\PP^1$-fibers by the map
$[x: y] \mapsto [x^{d(e)}: y^{d(e)}]$.
Here, we should regard $x$ and $y$ as restrictions of the
universal sections
\begin{equation*}
  x \in H^0\left(\mathcal R \otimes \OO_{C_e}(1)\right)
\end{equation*}
and
\begin{equation*}
  y \in H^0(\OO_{C_e}(1)).
\end{equation*}
They provide isomorphisms of line bundles
\begin{equation*}
  \OO([0]) \xrightarrow{\cong} \mathcal R \otimes \OO_{C_e}(1), \qquad
  \OO([\infty]) \xrightarrow{\cong} \OO_{C_e}(1),
\end{equation*}
where $[0]$ and $[\infty]$ denote the zero and infinity section of
$\mathcal C_e$, respectively.
The line bundle
$T_{\PP(\widetilde\LL_1^\vee|_{C_e} \oplus \OO_{\widetilde Z})/\mathcal
  C_e}$ corresponds to the sum of the zero and infinity section of
$\PP(\widetilde\LL_1^\vee|_{C_e} \oplus \OO_{\widetilde Z})$, and
therefore
$f^* T_{\PP(\widetilde\LL_1^\vee|_{C_e} \oplus \OO_{\widetilde
    Z})/\mathcal C_e} \cong \OO(d(e)[0] + d(e)[\infty])$.

From here, we can easily describe the contribution of the edge $e$.
There are no infinitesimal deformations of $C_e$ or the line bundle
$L_1|_{C_e}$, since $C_e$ is rational.
The deformations of the sections $\vec{x}$ are fixed, so together with
the automorphisms of $L_1|_{C_e}$ they contribute the virtual class,
which is simply $[\widetilde Z]$.
As for the deformations of $f$, we note that the space of sections of
$f^* T_{\PP(\widetilde\LL_1^\vee|_{C_e} \oplus \OO_{\widetilde
    Z})/\mathcal C_e}$ is spanned by a constant section, the sections
$x^a$ for $1 \le a \le d(e)$, and the sections $y^b$ for
$1 \le b \le d(e)$.
The fixed factor corresponding to the constant section cancels with
the automorphism factor corresponding to rescaling the curve.
The remaining part of
$R\pi_* (f^* T_{\PP(\widetilde\LL_1^\vee|_{C_e} \oplus \OO_{\widetilde
    Z})/\mathcal C})$ is moving, and we have
\begin{equation*}
  R\pi_* (f^* T_{\PP(\widetilde\LL_1^\vee|_{C_e} \oplus \OO_{\widetilde Z})/\mathcal C})^{\mathrm{mov}}
  = \prod_{a = 1}^{d(e)} \frac{a(\lambda - H)}{d(e)} \cdot \prod_{b = 1}^{d(e)} \frac{b(H - \lambda)}{d(e)}.
\end{equation*}

Now, suppose that $C_e$ has a special point at only one of its
ramification points.
Then $e$ is adjacent to an unstable vertex $v$, corresponding to a
basepoint of $\vec{x}$ of order $\beta(v)$.
The sections $\vec{x}$ still determine an underlying constant map to
$Z$, so the universal family over the edge moduli space is still given
by \eqref{universalfam}; the only difference in this situation is that
the universal map is now given by
$[x: y] \mapsto [x^{d(e) - \beta(v)}: y^{d(e)}]$, which is
well-defined by Remark~\ref{rem:dbeta}.

The normal directions to deformations of $\vec{x}$ are computed along
the lines of \cite{Bertram}.
Namely, the sections $\vec{x}$ must be of the form
\begin{equation*}
  \vec{x} = (c_1x^{\beta(v)}, \ldots, c_{N+1}x^{\beta(v)})
\end{equation*}
for constants $(c_1, \ldots, c_{N+1}) \in \C^{N+1} \setminus 0$, and
normal directions come from deforming the $x_i$ away from such
sections.
These contribute
\begin{equation}
  \label{Jdenom}
  \prod_{a=1}^{\beta(v)} \left(H + \frac{a}{d(e)}(\lambda-H)\right)^{N+1}
\end{equation}
to the Euler class of the virtual normal bundle.

There is also a part
\begin{equation*}
  \bigoplus_{i=1}^r R\pi_*\left(\mathcal{L}_1^{\otimes d_i}|_{C_e}\right)
  = \bigoplus_{i=1}^r \OO_{\widetilde Z}(d_i) \otimes R\pi_*\left(\OO_{\mathcal{C}_e}(d_i\beta(v))\right)
\end{equation*}
of the summand $\mathbb{E}$ in the obstruction
theory \eqref{pot2} corresponding to the obstructions to moving the
sections $\vec x$ away from $Z \subset \PP^N$.
Its moving part gives a contribution of
\begin{equation}
  \label{Jnum}
  \prod_{i=1}^r \prod_{a = 1}^{d_i\beta(v)} \left(d_i H + \frac{a}{d(e)} (\lambda-H)\right),
\end{equation}
to the virtual cycle.

Finally, a similar calculation to what we have done in the case of
edges without basepoints shows that the moving part of
$R\pi_* (f^* T_{\PP(\widetilde\LL_1|_{C_e} \oplus \OO_{\widetilde Z})/\mathcal C})$ contributes
\begin{equation*}
  \prod_{a = 1}^{d(e) - \beta(v)} \left(\frac{a}{d(e)}(\lambda-H)\right) \cdot \displaystyle \prod_{b=1}^{d(e)} \left(\frac{b}{d(e)}(H - \lambda)\right)
\end{equation*}
to the Euler class of the virtual normal bundle.

Observe that the quotient of \eqref{Jnum} by \eqref{Jdenom} is equal to the $q^{\beta(v)}$-coefficient in $z^{-1}J^{\epsilon}(q,z)$, under the substitution
\begin{equation}
\label{zsub}
z = \frac{\lambda - H}{d(e)}.
\end{equation}
Furthermore, in the twisted $\epsilon$-stable $J$-function, the twist
in the $q^{\beta(v)}$-coefficient is
\begin{align*}
  &e^{\C^*_{\lambda} \times \C^*_{z}}\left(R\pi_*(\mathcal{L}^{\vee}) \otimes \C_{(\lambda)}\right)^{-1}\\
  =&e^{\C^*_{\lambda} \times \C^*_{z}}\left(\big(\OO_Z(-1) \otimes \C_{(\lambda)}\big) \otimes-R\pi_*\left(\OO_{\mathcal C}(-\beta(v))\right) \right)\\
  =&\prod_{a=0}^{\beta(v) -1} (\lambda - H + az),
\end{align*}
which, under the substitution \eqref{zsub}, becomes
\begin{equation*}
  \prod_{a = d(e)-\beta(v)-1}^{d(e)} \left(\frac{a}{d(e)} (\lambda - H)\right).
\end{equation*}
Thus, we can express the entire contribution of an edge $e$ on which
$C_e$ has basepoints as
\begin{equation*}
  \frac{\left[\left.z^{-1}J^{\epsilon,\tw}\left(q,z\right)\right|_{z = \frac{\lambda - H}{d(e)}}\right]_{q^{\beta(v)}}}{\displaystyle \prod_{b=1}^{d(e)} \left(\frac{b}{d(e)}(\lambda - H)\right) \prod_{b=1}^{d(e)} \left(\frac{b}{d(e)}(H - \lambda)\right)},
\end{equation*}
where $[\cdot]_{q^{\beta}}$ denotes the $q^{\beta}$-coefficient of a power series in $q$.

Finally, we note that there are additional contributions to the virtual normal bundle from automorphisms of the edge components.  The only such automorphisms with nontrivial torus weight come from moving an unmarked ramification point, which occurs if $e$ is incident to an unstable vertex $v$ of valence $1$.  Such a vertex contributes
\begin{equation*}
  \frac{\lambda_{j(v)}}{d(e(v))}
\end{equation*}
to the inverse Euler class of the virtual normal bundle, where
\begin{equation*}
  \lambda_0 := \lambda - H, \qquad
  \lambda_\infty := H - \lambda
\end{equation*}
and $e(v)$ denotes the unique edge adjacent to $v$.

\subsubsection{Node contributions}

The deformations in $\mathcal{D}_{g,n,\beta}$ smoothing a node
contribute to the Euler class of the virtual normal bundle as the
first Chern class of the tensor product of the two cotangent line
bundles at the branches of the node.
For nodes at which a component $C_e$ meets a component $C_v$, this
contribution is
\begin{equation*}
  \frac{\lambda_{j(v)}}{d(e)} - \psi_v.
\end{equation*}
For nodes at which a component $C_e$ meets a component $C_{e'}$, the node-smoothing contribution is
\begin{equation*}
  \frac{\lambda_{j(v)}}{d(e)} + \frac{\lambda_{j(v)}}{d(e')},
\end{equation*}
where $v$ is the unstable vertex at which $e$ and $e'$ meet.
To ease notation, we combine the above two situations by writing the
contribution in either case as
\begin{equation*}
  -\psi - \psi',
\end{equation*}
where $\psi$ and $\psi'$ indicate the (equivariant) cotangent line
classes at the two branches of the node.

As for the node contributions from the normalization exact sequence,
each node $q$ (specified by a vertex $v$) contributes
\begin{equation}
  \label{node}
  e(R\pi_* (f^* T_{\PP(\LL_1^\vee \oplus \OO)/\mathcal C}|_q)) = \lambda_{j(v)}
\end{equation}
to the inverse Euler class of the virtual normal bundle.

\subsubsection{Total fixed-locus contribution}

Combining all of the above computations, we find that for a graph
$\Gamma$, the contribution
$\frac{[F_{\Gamma}]^{\vir}}{e(N_{\Gamma}^{\vir})}$ of $\Gamma$ to
$[\PP Z^{\epsilon}_{g, n, \beta, d}]^{\vir}$ equals:
\begin{align*}
  &\prod_{\substack{v \text{ stable}\\ j(v) =0}} [\M^{\epsilon}_{g(v),\val(v)}(Z,\beta(v))]^{\vir}_{\tw} \cdot \prod_{\substack{v \text{ stable}\\ j(v) =\infty}} \frac{[\M^{\infty}_{g(v),\val(v)}(Z,\beta(v))]^{\vir}}{e^{\C^*}\left(R\pi_*(\LL) \otimes \C_{(-\lambda)}\right)} \\
  \times & \prod_{e} \prod_{b=1}^{d(e)} \left(\frac{b}{d(e)}\lambda_0\lambda_\infty\right)^{-1} \times \prod_{\substack{v \text{ unstable}\\ (g(v), j(v)) = (0, \infty)}} \frac{\lambda_\infty}{d(e(v))} \\
  \times & \prod_{\substack{v \text{ unstable}\\ (g(v), j(v)) = (0, 0)}} \left[J^{\epsilon,\tw}\left(q,z\right)\bigg|_{z = \frac{\lambda_0}{d(e(v))}}\right]_{q^{\beta(v)}}\times \prod_{\text{nodes}} \frac{\lambda_{j(v)}}{-\psi - \psi'},
\end{align*}
Denoting the above by $\text{Contr}_{\Gamma}$, we have
\begin{equation*}
  [\PP Z^{\epsilon}_{g,n,\beta,d}]^{\vir} = \sum_{\Gamma} \frac{1}{|\text{Aut}(\Gamma)|} \iota_{\Gamma*} \left(\text{Contr}_{\Gamma}\right).
\end{equation*}

\section{Equivariant projective line}
\label{P1}

When $\beta=0$, torus localization on the twisted graph space is closely related to localization on the moduli space of stable maps to $\PP^1$.  In this section, we perform explicit computations of generating series related to localization on $\PP^1$, which play a role in the twisted graph space localization to follow.

First, note that the discussion in Section~\ref{fixedloci} can be specialized to the case when $N = 1$ and $r = 0$, in which case it recovers the Graber--Pandharipande localization formula for the Gromov--Witten theory of the projective line.

As in Section \ref{fixedloci}, the $\C^*$-fixed loci in
$\M_{g,n}(\PP^1, d)$ can be indexed by $n$-legged graphs $\Gamma$,
where each vertex $v$ is decorated by an index
$j(v) \in \{0, \infty\}$ and a genus $g(v)$, and each edge $e$ is
decorated by a degree $d(e) \in \mathbb{N}$.
Each vertex $v$ corresponds to a maximal sub-curve of genus $g(v)$
contracted to the single point $j(v) \in \PP^1$, or, in the unstable
case where the vertex has genus zero and valence one or two, to a
single point in the source curve.
Each edge $e$ corresponds to a noncontracted component, which is
necessarily of genus zero, and on which the map to $\PP^1$ is of the
form $[x:y] \mapsto [x^{d(e)}: y^{d(e)}]$ in coordinates.

Fix insertions $\alpha_1, \ldots, \alpha_n \in H^*_{\C^*}(\PP^1)$, and
let $p\colon \M_{g,n}(\PP^1, d) \to \M_{g, n}$ be the forgetful map.
Then the localization formula expresses the class
\begin{equation*}
  p_*\left(\prod_{i=1}^n \ev_i^*(\alpha_i) \cap [\M_{g,n}(\PP^1,d)]^{\vir}\right)
\end{equation*}
as a sum over contributions from each fixed-point graph $\Gamma$.
These expressions can be stated more efficiently by considering the generating series
\begin{equation}
  \label{int}
  \sum_{d = 0}^\infty y^d p_*\left(\prod_{i=1}^n \ev_i^*(\alpha_i) \cap [\M_{g,n}(\PP^1,d)]^{\vir}\right)
\end{equation}
for a Novikov variable $y$.

Let $\Phi_j$ denote the sum of all contributions to \eqref{int} from
graphs $\Gamma$ on which there is a vertex $v$ with $g(v) = g$ and
$j(v) = j$, and such that, after stabilization, the generic curve in
the moduli space corresponding to $\Gamma$ is smooth; the second condition means that there is no tree emanating from $v$
which contains more than one marking.
Therefore, emanating from the vertex $v$ on such a graph, there are
exactly $n$ trees containing a marking (which may be empty, if the marking lies on $v$) and $l$
trees with no marking, for some integer $l$.
It follows that
\begin{equation}
  \label{Phim}
  \Phi_j = \sum_{l=0}^{\infty} \frac 1{l!} \pi_{l*}\left(e_j \prod_{k = 1}^n S_j(\alpha_k, \psi_k) \prod_{k = n+1}^{n+l} \epsilon_j(\psi_k)\right),
\end{equation}
in which $\pi_l\colon \M_{g, n + l} \to \M_{g, n}$ denotes the forgetful map.  Here,
\begin{equation*}
  e_j := \sum_{i = 0}^g c_i(\mathbb E) (\sigma(j) \lambda)^{g - 1 - i},
\end{equation*}
where $\mathbb E$ is the Hodge bundle and
\begin{equation*}
  \sigma(j) :=
  \begin{cases}
    1 & \text{if } j = 0, \\
    -1 & \text{if } j = \infty.
  \end{cases}
\end{equation*}
The series $S_j(\alpha, z)$ in \eqref{Phim} is the universal
generating series of localization contributions from trees emanating
from a vertex $v$ with $j(v) = j$ that contains exactly one of the
markings and has an insertion of $\alpha \in H_{\C^*}^*(\PP^1)$, and
the series $\epsilon_j(z)$, similarly, is the generating series of
localization contributions of a tree containing none of the markings.

Let $\overline\psi_k$ be the pullback under $\pi_l$ of the class $\psi_k$ on $\M_{g, n}$.
It is well-known that $\overline\psi_k$ differs from $\psi_k$ exactly
on the boundary divisors of $\M_{g, n}$ where the $k$th marking and some of
the last $l$ markings lie on a rational tail.
By rewriting the classes $\psi_1, \dotsc, \psi_n$ in terms of
$\overline\psi_1, \dotsc, \overline\psi_n$ and boundary divisors, and
for each summand integrating along the fibers of the map forgetting
all markings of the involved boundary divisors, we can rewrite
$\Phi_j$ in the form
\begin{equation*}
  \Phi_j = \sum_{l=0}^{\infty} \frac 1{l!} \pi_{l*}\left(e_j \prod_{k = 1}^n \widetilde S_j(\alpha_k, \overline\psi_k) \prod_{k = n+1}^{n+l} \epsilon_j(\psi_k)\right)
\end{equation*}
for modified universal series $\widetilde S_j(\alpha, z)$.
Surprisingly, the series $\widetilde S_j(\alpha, z)$ is easier to
compute than $S_j(\alpha, z)$.
In fact, it is closely related to the $R$-matrix for the equivariant
Gromov--Witten theory of $\PP^1$.

For the proof of Theorem~\ref{thm:main}, what is relevant from this
formula is $\widetilde S_j(\alpha, 0)$.
In fact, this contribution is related to the idempotents in the
quantum cohomology ring of $\PP^1$.
Recall that the equivariant quantum cohomology of $\PP^1$ is isomorphic
to
\begin{equation*}
  \C[\lambda, H][\![y]\!]/(H(H - \lambda) - y),
\end{equation*}
where $H \in H^*_{\C^*}(\PP^1)$ is the equivariant hyperplane class.
It is easy to check that this ring is semisimple with idempotents
$e_0, e_\infty$ given by
\begin{equation*}
  e_0 = \frac{-\lambda/2 + (\lambda/2) \sqrt\phi + H}{\lambda \sqrt\phi}, \qquad
  e_\infty = \frac{\lambda/2 + (\lambda/2) \sqrt\phi - H}{\lambda \sqrt\phi},
\end{equation*}
where
\begin{equation*}
  \phi := 1 + \frac{4y}{\lambda^2}.
\end{equation*}

\begin{lemma}
  \label{lem:P1}
  We have the identities
  \begin{align*}
    \widetilde S_0(\one_{\PP^1}, 0) = \widetilde S_\infty(\one_{\PP^1}, 0) = \phi^{-1/4}, \\
    \widetilde S_0(H, 0) = \phi^{-1/4}\left(\frac\lambda 2 + \frac\lambda 2 \sqrt\phi\right), \\
    \widetilde S_\infty(H, 0) = \phi^{-1/4}\left(\frac\lambda 2 - \frac\lambda 2 \sqrt\phi\right).
  \end{align*}
\end{lemma}
As a consequence of this lemma, we have
\begin{equation}
  \label{fi0}
  \frac{\widetilde S_0(H, 0)}{\widetilde S_0(\one_{\PP^1}, 0)} = \frac\lambda 2 + \frac\lambda 2 \sqrt\phi.
\end{equation}
In particular, we note that this quantity is not a rational function of $\lambda$; this observation plays a key role in the proof of Theorem~\ref{thm:main}.

\begin{proof}[Proof of Lemma \ref{lem:P1}]
  Let $\omega_{g, n}(\alpha_1, \dotsc, \alpha_n)$ be the part of
  \eqref{int} in cohomological degree zero.
There are two ways of computing $\omega_{g, n}$: on the one hand, we can employ virtual localization, but on the other hand, we also know that the multilinear forms $\omega_{g, n}$
  form a topological field theory, which is determined by the
  equivariant Poincar\'{e} pairing and the quantum product.

  We first consider the computation of the degree-zero part of
  \eqref{int} via localization.
  In fact, $\omega_{g, n}(\alpha_1, \dotsc, \alpha_n)$ equals the
  degree-zero part of $\Phi_0 + \Phi_\infty$, because all localization
  graphs which do neither contribute to $\Phi_0$ or $\Phi_\infty$ give
  rise to a contribution supported on a nontrivial stratum of
  $\M_{g, n}$.
  Using that the Hodge bundle has rank $g$, we see that
  $\omega_{g, n}(\alpha_1, \dotsc, \alpha_n)$ equals the degree zero
  part of
  \begin{equation}
    \label{eq:p1loc}
    \sum_{j \in \{0, \infty\}} \sum_{l=0}^{\infty} \frac 1{l!} \pi_{l*}\left((\sigma(j) \lambda)^{g - 1} \prod_{k = 1}^n \widetilde S_j(\alpha_k, \overline\psi_k) \prod_{k = n+1}^{n+l} \epsilon_j(\psi_k)\right).
  \end{equation}
  By repeated application of the string equation for the last
  $l$ arguments (moving along the string flow), one can rewrite this
  in the form
  \begin{equation*}
    \sum_{j \in \{0, \infty\}} \sum_{l=0}^{\infty} \frac 1{l!} \pi_{l*}\left((\sigma(j) \lambda)^{g - 1} \prod_{k = 1}^n \widetilde S_j(\alpha_k, \overline\psi_k) \prod_{k = n+1}^{n+l} \te_j(\psi_k)\right),
  \end{equation*}
  where $\te_j$ is a new universal series such that $\te_j(0) = 0$.
  So,
  \begin{multline*}
    \omega_{g, n}(\alpha_1, \dotsc, \alpha_n) \\
    = \sum_{j \in \{0, \infty\}} \sum_{l=0}^{\infty} \frac 1{l!} \pi_{l*}\left((\sigma(j) \lambda)^{g - 1} \prod_{k = 1}^n \widetilde S_j(\alpha_k, \overline\psi_k) \prod_{k = n+1}^{n+l} \te_j^1 \psi_k\right),
  \end{multline*}
  where $\te_j^1$ is the linear coefficient in $z$ of $\te_j$.  Finally, by the dilaton equation,
  \begin{equation}
    \label{eq:p1tftloc}
    \omega_{g, n}(\alpha_1, \dotsc, \alpha_n) = \sum_{j \in \{0, \infty\}} (\sigma(j) \lambda)^{g - 1} (1 - \te_j^1)^{-(2g - 2 + n)} \prod_{i = 1}^n \widetilde S_j(\alpha_k, 0).
  \end{equation}

  It is useful to note that the tree series satisfy
  \begin{equation}
    \label{eq:p1limite}
    \epsilon_j|_{y = 0} = \te_j|_{y = 0} = 0,
  \end{equation}
  since a localization tree containing no marking needs to carry a
  positive degree.
  Since for $y = 0$ the quantum idempotents $e_0$ and $e_\infty$
  recover the classical idempotents, and the only contributions to
  $S_0$ or $S_\infty$ in degree zero come from empty trees, we also
  have
  \begin{align}
    \label{eq:p1limitS}
    &\widetilde S_0(e_0, 0)|_{y = 0} = \widetilde S_\infty(e_\infty, 0)|_{y = 0} = 1,\\
    \nonumber &\widetilde S_0(e_\infty, 0)|_{y = 0} = \widetilde S_\infty(e_0, 0)|_{y = 0} = 0.
  \end{align}

  We now compare \eqref{eq:p1tftloc} to explicit expressions of the
  topological field theory.
  First, consider the case $n = 0$.
  Since the norms of the idempotents $e_0$ and $e_\infty$ are given by
  $\lambda^{-1} \phi^{-1/2}$ and $-\lambda^{-1} \phi^{-1/2}$,
  respectively, we have
  \begin{equation*}
    \omega_{g, 0}() = \lambda^{g - 1} \phi^{\frac{g - 1}2} + (-\lambda)^{g - 1} \phi^{\frac{g - 1}2}.
  \end{equation*}
  Comparing this to \eqref{eq:p1tftloc} (note that both hold for any
  $g$) and using \eqref{eq:p1limite}, we see that we must have
  \begin{equation*}
    1 - \te_0^1 = 1 - \te_\infty^1 = \phi^{-\frac 14}.
  \end{equation*}

  The trilinear form $\omega_{0, 3}$ is given by quantum
  multiplication and application of the equivariant Poincar\'{e} pairing.
  It is easy to compute that
  \begin{equation*}
    \omega_{0, 3}(e_{i_1}, e_{i_2}, e_{i_3}) = \delta_{i_1 = i_2 = i_3} \sigma(i_1) \lambda^{-1} \phi^{-\frac 12}.
  \end{equation*}
  With \eqref{eq:p1tftloc}, this gives the equations
  \begin{equation*}
    \begin{aligned}
      \lambda^{-1} \phi^{\frac 14} (\widetilde S_0(e_0, 0))^3 - \lambda^{-1} \phi^{\frac 14} (\widetilde S_\infty(e_0, 0))^3 &= \lambda^{-1} \phi^{-\frac 12} \\
      \lambda^{-1} \phi^{\frac 14} (\widetilde S_0(e_0, 0))^2 (\widetilde S_0(e_\infty, 0)) - \lambda^{-1} \phi^{\frac 14} (\widetilde S_\infty(e_0, 0))^2 (S_\infty e_\infty) &= 0 \\
      \lambda^{-1} \phi^{\frac 14} (\widetilde S_0(e_0, 0)) (\widetilde S_0(e_\infty, 0))^2 - \lambda^{-1} \phi^{\frac 14} (\widetilde S_\infty(e_0, 0)) (\widetilde S_\infty(e_\infty, 0))^2 &= 0 \\
      \lambda^{-1} \phi^{\frac 14} (\widetilde S_0(e_\infty, 0))^3 - \lambda^{-1}
      \phi^{\frac 14} (\widetilde S_\infty(e_\infty, 0))^3 &= -\lambda^{-1}
      \phi^{-\frac 12}.
    \end{aligned}
  \end{equation*}
  
  It is not difficult to see that the only solutions to these
  equations together with \eqref{eq:p1limitS} are
  \begin{equation*}
    \widetilde S_0(e_0, 0) = \widetilde S_\infty(e_\infty, 0) = \phi^{-\frac 14}, \qquad \widetilde S_0(e_\infty, 0) = \widetilde S_\infty(e_0, 0) = 0.
  \end{equation*}
With the identities
  \begin{equation*}
    \one_{\PP^1} = e_0 + e_\infty, \qquad H = \left(\frac\lambda 2 + \frac\lambda 2 \sqrt\phi\right) e_0 + \left(\frac\lambda 2 - \frac\lambda 2 \sqrt\phi\right) e_\infty,
  \end{equation*}
  the lemma easily follows.
\end{proof}

\section{Proof of Theorem \ref{thm:main}}
\label{sec:proof}

We are now ready to turn to the proof of the main theorem.
The basic structure of the proof is to compute the difference between the expression
\begin{equation}
  \label{4}
  \sum_{\beta,d} q^{\beta}y^d p_*\left( \ev_1^*(\alpha) \cap [\PP Z^\epsilon_{g, n, \beta,d}]^{\vir} \right)
\end{equation}
and the expression
\begin{multline}
  \label{5}
 \sum_{d=0}^{\infty} p_*\left(\sum_{k=0}^{\infty}\right. \sum_{\substack{\beta_0, \beta_1, \ldots, \beta_k\\ d_1, \ldots, d_k\\ d_i \leq \beta_i \; \forall i}}\frac{q^{\beta_0}y^{d} b_{(\beta_1, \ldots, \beta_k)*}}{k!} \bigg(\ev_1^*(\alpha) \; \cup \\
  \prod_{i = 1}^k q^{\beta_i}y^{d_i} \ev_{n + i}^*( \widetilde{\mu}_{\beta_i,d_i}^{\epsilon}(-\psi_{n + i})) \cap [\PP Z^{\infty}_{g, n+k, \beta_0,d}]^{\vir} \left.\bigg)\right),
\end{multline}
by localization on the twisted graph space.  Here,
\begin{equation*}
  p\colon \PP Z^\epsilon_{g, n, \beta,d} \rightarrow \M^{\epsilon}_{g, n}(Z, \beta)
\end{equation*}
is the morphism forgetting $L_2, z_1$, and $z_2$ (and stabilizing as
necessary).  The insertion $\alpha$ is an element of $H^*(\PP Z)$, which is isomorphic as a vector space to $\mathcal{H} \otimes H^*(\PP^1)$.  The mirror transformation $\widetilde{\mu}^{\epsilon}$ of the twisted graph space is defined by
\begin{equation*}
\sum_{d=0}^{\infty} y^d \widetilde{\mu}^{\epsilon}_{\beta,d}(z) = \mu^{\epsilon,\text{tw}}_{\beta}(z) \otimes \varphi_0 - \sum_{d=1}^{\beta} y^{d} \frac{\lambda_\infty \cdot [J^{\epsilon}_\tw(q, \frac{\lambda_0}{d})]_{q^{\beta}} }{(\frac{\lambda_\infty}{d} + z) \prod_{i=1}^{d} \prod_{j \in \{0, \infty\}} \frac{i\lambda_j}{d}} \otimes \varphi_{\infty},
\end{equation*}
where
\begin{equation*}
\varphi_0 := \frac{[0]}{\lambda}, \qquad \varphi_{\infty}:= -\frac{[\infty]}{\lambda},
\end{equation*}
for $[0]$ and $[\infty]$ the classes of the $0$- and $\infty$-sections in $\PP Z$.

As we show in Section~\ref{subsec:mainproof} below, whenever the twisted wall-crossing Theorem~\ref{thm:twisted} holds, the localization contributions to \eqref{4} and \eqref{5} agree.  More importantly, though, we show that Theorem \ref{thm:twisted} can be deduced using only the fact that the difference between \eqref{4} and \eqref{5} is a Laurent polynomial in the equivariant parameter.  By Lemma \ref{twuntw}, this implies
Theorem \ref{thm:main}.

\subsection{Multilinear forms}

First, we set up some useful notation and observations.

Define
\begin{equation*}
  \left(\;\right)_{g, n, \beta}^\epsilon\colon \HH[\![\psi]\!]^n \to H_*(\M^{\epsilon}_{g, n}(Z,\beta))  
\end{equation*}
by
\begin{equation*}
  \left(\phi_1, \dotsc, \phi_n\right)_{g, n, \beta}^\epsilon
  = \prod_{i=1}^n \ev_i^*(\phi_i)(\psi_i) \cap [\M^{\epsilon}_{g, n}(Z,\beta)]^{\vir}_{\tw},  
\end{equation*}
define
\begin{equation*}
  \left(\;\right)_{g, n, \beta}^{\infty \to \epsilon}\colon \HH[\![\psi]\!]^n \to H_*(\M^{\epsilon}_{g, n}(Z,\beta))  
\end{equation*}
by
\begin{multline*}
  \sum_\beta q^\beta \left(\phi_1, \dotsc, \phi_n\right)_{g, n,
    \beta}^{\infty \to \epsilon}
  = \sum_{k, \beta_0, \beta_1, \ldots, \beta_k} \frac{q^{\beta_0}}{k!} b_{\bbeta*} c_* \Bigg(\prod_{i=1}^n \ev_i^*(\phi_i)(\psi_i) \\
  \cup \prod_{i = 1}^k q^{\beta_i} \ev_{n +
    i}^*(\mu_{\beta_i}^{\epsilon, \tw}(-\psi_{n + i})) \cap [\M^{\infty}_{g,
    n+k}(Z,\beta_0)]^{\vir}_{\tw} \Bigg),
\end{multline*}
and define
\begin{equation*}
  \left(\;\right)_{g, n, \beta}^{\text{WC}}\colon \HH[\![\psi]\!]^n \to H_*(\M^{\epsilon}_{g, n}(Z,\beta))  
\end{equation*}
by
\begin{equation*}
  \left(\phi_1, \dotsc, \phi_n\right)_{g, n, \beta}^{\text{WC}} = \left(\phi_1, \dotsc, \phi_n\right)_{g, n, \beta}^\epsilon - \left(\phi_1, \dotsc, \phi_n\right)_{g, n, \beta}^{\infty \to \epsilon}.  
\end{equation*}

Then Theorem~\ref{thm:twisted} is equivalent to the statement that the form $\left(\;\right)_{g, n, \beta}^{\text{WC}}$ vanishes when all insertions lie in $\HH \subset \HH[\![\psi]\!]$.  In fact, we need only verify this vanishing when $\phi_i = \one$ for each $i$:

\begin{lemma}
  \label{lem:red}
  The equation
  \begin{equation*}
    \left(\phi_1, \dotsc, \phi_n\right)_{g, n, \beta}^{\text{WC}} = 0
  \end{equation*}
  holds for any $\phi_1, \dotsc, \phi_n \in \HH$ if and only if
  \begin{equation}
  \label{eq:1}
    \left(\one, \dotsc, \one\right)_{g, n, \beta}^{\text{WC}} = 0.
  \end{equation}
  Thus, Theorem~\ref{thm:main} is equivalent to \eqref{eq:1}.
\end{lemma}
\begin{proof}
  This follows from the fact that the evaluation maps at the first $n$
  marked points commute with $b_{\bbeta}$, as well as the fact that
  the $i$th cotangent line bundle is preserved under pullback by
  $b_{\bbeta}$.
\end{proof}

\subsection{Proof of Theorem~\ref{thm:twisted}}
\label{subsec:mainproof}

We prove the vanishing of $\left(\one, \dotsc, \one\right)_{g, n, \beta}^{\text{WC}}$, and thus Theorem~\ref{thm:twisted}, by a series of inductions.

The first induction is on the degree $\beta$.  When $\beta = 0$, there is nothing to prove, since
$\M^{\epsilon}_{g, n}(Z, 0) = \M^\infty_{g, n}(Z, 0)$.  Fix a degree $\beta$, then, and suppose that $\left(\one, \dotsc, \one\right)_{g, n, \beta'}^{\text{WC}} = 0$ has been shown for all $\beta' < \beta$ (and all $g$ and $n$).

The next induction is on the genus $g$.  When $g = 0$, Theorem~\ref{thm:twisted} was proven by Ciocan-Fontanine and Kim in \cite{CFKZero}.  (They state their result only on the level of invariants, but it can be readily generalized to the level of virtual cycles.  See also \cite{CladerRoss} for the analogous genus-zero virtual cycle result in the Landau--Ginzburg chamber.)  Thus, we fix a genus $g$ and suppose that $\left(\one, \dotsc, \one\right)_{g', n, \beta}^{\text{WC}} = 0$ has been shown for all $g' < g$ (and all $n \geq 1$).

Denote by
\begin{equation*}
  D_{\beta}(\lambda,y) \in H_*(\M^{\epsilon}_{g, n}(Z, \beta))[\lambda,
  \lambda^{-1}][\![y]\!]
\end{equation*}
the coefficient of $q^{\beta}$ in the difference between \eqref{4} and \eqref{5}.  Both \eqref{4} and \eqref{5} can be computed by localization on the
respective twisted graph spaces, and the graphs indexing fixed loci appearing in \eqref{4} and
the graphs indexing fixed loci appearing in \eqref{5} differ in only
two ways: the latter have $k$ additional legs, and they have smaller total degree $\beta_0$.  Recall also that localization graphs need to satisfy $d(e) > \beta(v)$ for each genus-zero, valence-one vertex $v$ with unique incident edge $e$ (see Remark~\ref{rem:dbeta}).  If a graph does not necessarily satisfy this degree condition, but it satisfies all other conditions of a localization graph, then we refer to it as a ``fake localization graph".

For each choice of $k$, each choice of a contributing localization graph $\Gamma_k$ to \eqref{5}, and each choice of $\vec{\beta} = (\beta_1, \ldots, \beta_k)$ and $\vec{d} = (d_1, \ldots, d_k)$, define a fake localization graph $\Gamma$ for \eqref{4} by applying the following operations:
\begin{itemize}
\item Remove all additional legs at vertices $v$ with $j(v) = 0$ and add their $\beta$-degree to the incident vertex.
\item Replace each additional leg at a vertex $v$ with $j(v) = \infty$ by a new edge connected to a new genus-zero, valence-one vertex, where the degree of the new edge is prescribed by $\vec{d}$ and the $\beta$-degree of the former extra leg is put on the new vertex.
\end{itemize}
We refer to the sum of all contributions of localization graphs $\Gamma_k$ corresponding to the fake localization graph $\Gamma$ as ``the contribution of $\Gamma$ to \eqref{5}".  In fact, if $v$ is a genus-zero, valence-one vertex of a fake localization graph $\Gamma$ with unique incident edge $e$, then the definition of $\widetilde{\mu}^{\epsilon}$ is exactly chosen such that if $d(e) \leq \beta(v)$, then the local contribution of $v$ to the localization formula for \eqref{5} vanishes.  Thus, since fake localization graphs do not appear in \eqref{4}, we can express both \eqref{4} and \eqref{5}---and hence $D_{\beta}(\lambda,y)$---as a sum over the honest localization graphs for \eqref{4}.

The contribution to $D_{\beta}(\lambda,y)$ of many such graphs can immediately be shown to vanish.  Indeed, suppose that $\Gamma$ is a localization graph in which there are at least two vertices of positive degree. For any vertex $v$ such that $j(v) = \infty$, the contributions of $v$
to \eqref{4} and \eqref{5} are identical, and the contributions from
each edge are also the same. 
Thus, the only possibly non-identical contributions come from vertices
$v$ with $j(v) = 0$.

Let us first consider vertices $v$ of $\Gamma$ with $j(v) = 0$ that are unstable of genus zero and valence one, with unique incident edge $e$ of degree $d(e) > \beta(v)$.  The contribution of such a vertex to the localization formula for
\eqref{4} is given by the coefficient of $q^{\beta(v)}$ in
\begin{equation*}
  J^{\epsilon,\tw}(q, z)\bigg|_{z = \frac{\lambda_0}{d(e)}},
\end{equation*}
where we again use the notation $\lambda_0 := \lambda - H$.
The corresponding contribution to \eqref{5} (after forgetting about
the additional markings at $v$) is given by
\begin{multline*}
  \delta_{\beta(v), 0} \frac{\lambda_0}{d(e)} + \mu_{\beta(v)}^{\epsilon, \tw}\left(\frac{\lambda_0}{d(e)}\right) \\
  + \sum_{\substack{k \ge 1, \beta_0 + \dotsb + \beta_k = \beta(v) \\ k + \beta_0 > 1}} \ev_{1, *} \left(\frac{\prod\limits_{i = 1}^k \ev_{i + 1}^* \left(\mu_{\beta_i}^{\epsilon, \tw}(-\psi_{i + 1})\right)}{\left(\frac{\lambda_0}{d(e)} - \psi_1\right) k!} \cap [\M^\infty_{0, 1 + k}(Z, \beta_0)]^{\vir}_\tw\right) \\
  = J^{\infty,\tw}\left(q, \sum_\beta q^\beta \mu_\beta^{\epsilon, \tw}(-z), z\right)\bigg|_{z = \frac{\lambda_0}{d(e)}},
\end{multline*}
where the first two summands come from the cases where the vertex corresponding to $v$ in \eqref{5} is
unstable and the last summand comes from the case where this vertex is
stable.  By the $J$-function wall-crossing \eqref{JWC}, these contributions
to \eqref{4} and \eqref{5} are identical.

For stable vertices with $j(v) = 0$, on the other hand, the twisted wall-crossing holds in degree $\beta(v) < \beta$ by induction.  This says precisely that the contributions to \eqref{4} and \eqref{5} from such vertices also agree.

Thus,\footnote{Here, we are using the splitting property satisfied
  by the quasimap virtual fundamental class (see
  \cite[Section~2.3.3]{CFKHigher}).} we have shown that the
contributions to \eqref{4} and \eqref{5} agree whenever there are two
vertices of positive degree.  Similarly, our induction on the genus shows that the contributions to \eqref{4} and \eqref{5} agree whenever there are two vertices of positive genus.  Thus, we can express $D_{\beta}(\lambda,y)$ as a sum over contributions from graphs
$\Gamma$ in which there is a vertex $v_0$ with
$\beta(v_0) = \beta$ and $g(v_0) = v$, and all other vertices have degree and genus zero.
(Note that in the contributions to \eqref{5} from such a graph, the
last $k$ markings must all lie at $v_0$.)  Similarly to the situation in Section~\ref{P1}, there are $m$ trees
emanating from $v_0$ on which at least one of the markings
$q_1, \ldots, q_n$ lies, and $l$ unmarked trees.

We now begin the third induction, on the number $n$ of marked points.  Suppose that $\left(\one, \dotsc, \one\right)_{g, n', \beta}^{\text{WC}} = 0$ holds for all $n' < n$.  Then the contribution of $\Gamma$ to $D_{\beta}(\lambda,y)$ vanishes unless each of the markings $q_1, \ldots, q_n$ lies on a separate (possibly empty) tree emanating from $v_0$.

Note that since $\beta(v) = 0$ for all vertices of the trees, the
localization contribution of such trees is identical to those
discussed in Section~\ref{P1}, except that we should replace $\lambda$
by $\lambda_0$.
Thus, if
\begin{equation*}
  \pi_l\colon \M^{\epsilon}_{g, n + l}(Z, \beta) \rightarrow \M^{\epsilon}_{g, n}(Z, \beta)
\end{equation*}
denotes the forgetful map, $D_{\beta}(\lambda, y)$ is expressed as
\begin{equation*}
  D_{\beta}(\lambda, y) = \sum_{l=0}^{\infty} \frac{1}{l!} \pi_{l*}\left(g_1(\psi), \dotsc, g_n(\psi), \epsilon^{\lambda_0}(\psi), \dotsc, \epsilon^{\lambda_0}(\psi) \right)_{g, n + l, \beta}^{\text{WC}}
\end{equation*}
where
\begin{equation*}
  g_i(z) =
  \begin{cases}
    S_0^{\lambda_0}(\alpha, z) & \text{if }i = 1, \\
    S_0^{\lambda_0}(\one, z) & \text{if }i > 1, \\
  \end{cases}
\end{equation*}
and $\epsilon^{\lambda_0}(z)$ (respectively
$S_0^{\lambda_0}(\phi, z)$) is obtained from $\epsilon_0(z)$
(respectively $S_0(\phi, z)$) by replacing $\lambda$ by $\lambda_0$.

As in Section~\ref{P1}, we rewrite this by expressing the
$\psi$-classes at the first $n$ markings in terms of the
$\psi$-classes $\overline{\psi}_j$ pulled back under
$\pi_l$.
Note that the classes $\psi_j$ and $\overline{\psi_j}$ differ exactly
on the locus of curves with a rational component of degree $\beta = 0$
containing marking $j$ and some of the markings
$n + 1, \dotsc, n + l$.
Hence, expressing the $\psi$-classes in terms of the pullback
$\psi$-classes formally works in the same way as for $\M_{g, n}$, and
one obtains
\begin{equation*}
  D_{\beta}(\lambda, y) = \sum_{l = 0}^\infty \frac 1{l!}
  \pi_{l*}\left(\tg_1(\overline\psi), \dotsc, \tg_n(\overline\psi), \epsilon^{\lambda_0}(\psi), \dotsc, \epsilon^{\lambda_0}(\psi)\right)_{g, n + l, \beta}^{\text{WC}},
\end{equation*}
where $\tg_i(z)$ is defined in the same way as $g_i(z)$ but with $S_0$
replaced by $\widetilde S_0$, and $\overline\psi^k$ is a symbol that
should be replaced by the pullback $\psi$-class $\overline{\psi}_j^k$
in position $j$.

We define
\begin{equation*}
  U_{l, k} \in \HH^{\otimes l}[\lambda_0, \lambda_0^{-1}][\![y, \psi]\!]
\end{equation*}
such that
\begin{equation*}
  \sum_{k = -l}^\infty U_{l, k}
  = \sum_{l = 0}^\infty \frac 1{l!} \epsilon^{\lambda_0}(\psi) \otimes \dotsb \otimes \epsilon^{\lambda_0}(\psi) \in \HH^{\otimes l}[\lambda_0, \lambda_0^{-1}][\![y, \psi]\!],
\end{equation*}
and such that $U_{l, k}$ is homogeneous of degree $k + l$ for the
grading where $\deg(y) = \deg(\lambda) = 0$ and $\deg(H) = \deg(\overline\psi) = 1$.
So, slightly abusing notation, we can write
\begin{equation}
  \label{l'}
  D_{\beta}(\lambda, y) = \sum_{l = 0}^\infty \sum_{k = -l}^\infty
  \pi_{l*}\left(\tg_1(\overline\psi), \dotsc, \tg_n(\overline\psi), U_{l, k}\right)_{g, n + l, \beta}^{\text{WC}}.
\end{equation}

We now claim that, for any non-positive $k$, the wall-crossing formula
\begin{equation}
  \label{eq:k}
 \sum_{l = 0}^\infty \pi_{l*} \left(\one, \dotsc, \one, U_{l, k}\right)_{g, n + l, \beta}^{\text{WC}} = 0
\end{equation}
holds.

The proof of \eqref{eq:k} is by induction on $k$, using that \eqref{l'} is a
Laurent polynomial in $\lambda$ for any choice of the insertion
$\alpha$.
The claim is trivially true for $k$ sufficiently small, since the left-hand side of \eqref{eq:k} lives in
\begin{equation*}
  H_{2(\vdim(\M^\epsilon_{g, n}(Z, \beta)) - k)}(\M^\epsilon_{g, n}(Z, \beta)),
\end{equation*}
which is trivial when $\vdim(\M^\epsilon_{g, n}(Z, \beta)) - k$ is
bigger than the dimension of all components of
$\M^\epsilon_{g, n}(Z, \beta)$.

Now, suppose the claim holds for all $k < k_0$.
This implies that all summands of \eqref{l'} with $k < k_0$ vanish, since $\one$ differs from $\tg_j(\overline\psi_j)$ only by a factor
pulled back from $\M^\epsilon_{g, n}(Z, \beta)$.
It follows that the highest possible dimension of a summand is
\begin{equation}
  \label{eq:vdim}
  \dim\M^\epsilon_{g, n}(Z, \beta) - k_0,
\end{equation}
and this dimension is achieved only in the $k = k_0$ term from taking
the coefficient of $z^0 H^0$ for each $\tg_j(z)$.
It follows that
\begin{equation}
  \label{eq:topdim}
  \sum_{l = 0}^\infty
  \pi_{l*}\left(\widetilde S_0(\alpha, 0), \widetilde S_0(\one, 0), \dotsc, \widetilde S_0(\one, 0), U_{l, k_0}\right)_{g, n + l, \beta}^{\text{WC}}
\end{equation}
is a Laurent polynomial in $\lambda$.

The crucial point, now, is that the above is true for either choice of
the insertion $\alpha \in \{\one, H\}$.
By \eqref{fi0}, changing $\alpha$ from $\one$ to $H$ causes
$\tg_{1}(0)$ to change by a factor of
\begin{equation}
  \label{difference}
  \frac{\lambda}2 + \frac{\lambda}2 \sqrt{1+ \frac{4y}{\lambda^2}}.
\end{equation}
Since \eqref{difference} has infinitely many negative powers when
expanded as a Laurent series in $\lambda$, this is only possible if
\eqref{eq:topdim} is identically zero.
Dividing \eqref{eq:topdim} for $\alpha = \one$ by $\phi^{-n/4}$ (see
Lemma~\ref{lem:P1}) completes the induction step for the proof of
\eqref{eq:k}.

\begin{remark}
  \label{rem:mp}
  Note that the assumption of one marked point is necessary at this
  stage of the proof, in order to vary the insertion $\alpha$.
  One might hope to remove this assumption, but in order for the
  unmarked wall-crossing to appear in the localization formula for the
  twisted graph space $\PP Z^{\epsilon}_{g, n, \beta, d}$, we would
  need a localization graph consisting of exactly one vertex $v$,
  which has $j(v) = 0$.
  Such a graph exists only when $d = 0$, but in this case the virtual
  dimension of the twisted graph space is $g + \beta - 1$ less than
  the virtual dimension of $\M^{\epsilon}_{g, n}(Z, \beta)$, so when
  $g > 0$ or $\beta > 0$ we cannot expect to obtain any relations from
  localization.
\end{remark}

Since $\epsilon^{\lambda_0}(\psi)|_{y = 0} = 0$, the $y^0$-coefficient
of \eqref{eq:k} for $k = 0$ gives exactly the wall-crossing formula. This completes the proof of Theorem~\ref{thm:twisted}.

\bibliographystyle{amsplain}

\providecommand{\bysame}{\leavevmode\hbox to3em{\hrulefill}\thinspace}
\providecommand{\MR}{\relax\ifhmode\unskip\space\fi MR }
\providecommand{\MRhref}[2]{%
  \href{http://www.ams.org/mathscinet-getitem?mr=#1}{#2}
}
\providecommand{\href}[2]{#2}

\end{document}